\documentclass[preprint,12pt]{elsarticle}
\biboptions{numbers,sort&compress}

\usepackage{amssymb}
\usepackage{amsthm}
\usepackage{amsmath}
\usepackage{epic}
\usepackage{CJK}
\usepackage{setspace}
\usepackage{bm}
\newtheorem{theorem}{Theorem}[section]
\newtheorem{corollary}[theorem]{Corollary}
\newtheorem{lemma}[theorem]{Lemma}

\newtheorem{example}[theorem]{Example}

\journal{~}

\begin{document}
\begin{spacing}{1.15}
\begin{CJK*}{GBK}{song}
\begin{frontmatter}
\title{\textbf{Spectral characterizations of local structures of graphs and hypergraphs}}

\author{Jiang Zhou}
\author{Changjiang Bu}

\address{School of Mathematical Sciences, Harbin Engineering University, Harbin 150001, PR China}

\begin{abstract}
In this paper, we give the relationship between spectral radius and local structures of graphs and hypergraphs. Our work shows that certain local subgraphs (subhypergraphs) must occur when the spectral radius ratio is large. We also give spectral bounds on the local vector chromatic number in terms of tensor eigenvalues of graphs.
\end{abstract}

\end{frontmatter}

\section{Introduction}
A hypergraph $H$ is called $r$-\textit{uniform} if each edge of $H$ contains exactly $r$ distinct vertices. Let $V(H)$ and $E(H)$ denote the vertex set and the edge set of $H$, respectively. The \textit{2-section graph} of a hypergraph $H$, denoted by $H^{(2)}$, is the graph on the vertex set $V(H)$, and $\{u,v\}\in E(H^{(2)})$ if and only if there exists $e\in E(H)$ such that $\{u,v\}\subseteq e$. A hypergraph $H_0$ is called a \textit{subhypergraph} of $H$ if $V(H_0)\subseteq V(H)$ and $E(H_0)\subseteq E(H)$. The \textit{adjacency tensor} \cite{Cooper} of an $r$-uniform hypergraph $H$, denoted by $\mathcal{A}_H=(a_{i_1i_2\cdots i_r})$, is an $r$-order $|V(H)|$-dimension tensor with entries
\begin{eqnarray*}
a_{i_1i_2\cdots i_r}=\begin{cases}\frac{1}{(k-1)!}~~~~~~~\mbox{if}~\{i_1,i_2,\ldots,i_r\}\in E(H),\\
0~~~~~~~~~~~~~\mbox{otherwise}.\end{cases}
\end{eqnarray*}
The \textit{spectral radius} of $H$, denoted by $\rho(H)$, is defined as the spectral radius of $\mathcal{A}_H$. Clearly, for a graph ($2$-uniform hypergraph) $G$, $\rho(G)$ is the spectral radius of the adjacency matrix of $G$. Let $\beta(H)=\frac{\rho(H)}{\rho(H^{(2)})}$ denote the \textit{spectral radius ratio} of $H$.

If a vertex subset $C$ of a graph $G$ induces a complete graph, then $C$ is called a clique in $G$. A clique with $r$ vertices is called an $r$-clique. Let $C_r(G)$ denote the set of $r$-cliques in $G$. The \textit{$r$-clique tensor} \cite{Liu} of $G$, denoted by $\mathcal{A}_r(G)=(a_{i_{1}i_{2}\cdots i_{r}})$, is an $r$-order $|V(G)|$-dimension tensor with entries
\begin{align*}
a_{i_{1}i_{2}\cdots i_{r}}=
\begin{cases}
   \frac{1}{(t-1)!}, &\{i_1,\ldots,i_r\}\in C_{t}(G).\\
   0, &\textup{otherwise}.
\end{cases}
\end{align*}
The spectral radius of $\mathcal{A}_r(G)$ is called the \textit{$r$-clique spectral radius} of $G$, denoted by $\rho_r(G)$. Let $\gamma_r(G)=\frac{\rho_r(G)}{\rho_2(H)}$, where $H$ is the subgraph of $G$ obtained by deleting all edges that are not containing in an $r$-clique. Actually, $\mathcal{A}_r(G)$ equals to the adjacency tensor of an $r$-uniform hypergraph $H_r$ with vertex set $V(H_r)=V(G)$ and edge set $E(H)=\{i_1\cdots i_r:\{i_1,\ldots,i_r\}\in C_r(G)\}$. Recent years, the research on tensor eigenvalues of graphs and hypergraphs have received extensive attention \cite{ChenBu,Clark,GaoChang,Liu_2024,YuPeng}.

The spectral extremal problems of graph and hypergraphs have been an active topic in spectral graph theory. From the research on spectral extremal problems, it is easy to see that certain subgraphs (subhypergraphs) must occur when the spectral radius or the number of edges is large enough. We will consider a localized problem: Are there some spectral conditions to guarantee the existence of certain local subgraphs (subhypergraphs) in the open neighborhood of some vertex?

The Lov\'{a}sz theta function \cite{Lovasz} $\vartheta(G)$ is a powerful tool for studying the independence number and Shannon capacity of graphs. The Schrijver theta function \cite{Schrijver} $\vartheta'(G)$ is a smaller upper bound for the independence number $\alpha(G)$, that is, $\alpha(G)\leq\vartheta'(G)\leq\vartheta(G)$. Based on vector coloring, the vector chromatic number $\chi_v(G)$ was introduced in \cite{Karger}, and it is known \cite{Balla} that $\chi_v(G)=\vartheta'(\overline{G})\leq\chi(G)$, where $\overline{G}$ is the complement of $G$, $\chi(G)$ denotes the chromatic number of $G$. An eigenvalue bound on $\chi_v(G)$ was given in \cite{Bilu}. We will investigate spectral bounds on the vector chromatic number in local induced graphs.

This paper is organized as follows. In Section 2, some definitions, notations and auxiliary lemmas are introduced. In Sections 3 and 4, we give the relationship between spectral radius and local structures of graphs and hypergraphs. We show that certain subgraphs (subhypergraphs) must occur in the open neighborhood of some vertex when $\gamma_r(G)$ ($\beta(H)$) is large enough.

In Section 5, we give spectral lower bounds on the local vector chromatic number in terms of tensor eigenvalues of graphs, which are also spectral lower bounds on the local chromatic number introduced by Erd\H{o}s et al. \cite{ErdosFuredi}.

\section{Preliminaries}

An $m$-order $n$-dimension complex tensor $\mathcal{A}=(a_{i_1\cdots i_m})$ is a multidimensional array with $n^{m}$ entries, where $i_{j}=1,2,...,n$, $j=1,2,...,m$. For a vector $x=(x_{1},\ldots,x_{n})^\top\in\mathbb{C}^{n}$, let $\mathcal{A}x^{m-1}$ denote a vector in $\mathbb{C}^{n}$ whose $i$-th component is $\sum_{i_{2},\ldots,i_{m}=1}^{n}a_{ii_{2}\cdots i_{m}}x_{i_{2}}\cdots x_{i_{m}}$ (see \cite{Qi}).
If there exist $\lambda\in\mathbb{C}$ and nonzero vector $x=(x_{1},\ldots,x_{n})^\top\in\mathbb{C}^{n}$ satisfying
\begin{align*}
\mathcal{A}x^{m-1}=\lambda x^{[m-1]},
\end{align*}
then $\lambda$ is called an \textit{eigenvalue} of $\mathcal{A}$, and $x$ is called an \textit{eigenvector} of $\mathcal{A}$ corresponding to $\lambda$, where $x^{[m-1]}=(x_{1}^{m-1},\ldots,x_{n}^{m-1})^\top$. The largest modulus of all eigenvalues of $\mathcal{A}$ is called the spectral radius of $\mathcal{A}$, denoted by $\rho(\mathcal{A})$. When $\mathcal{A}$ is a nonnegative tensor, $\rho(\mathcal{A})$ is an eigenvalue of $\mathcal{A}$, and there exists a nonnegative eigenvector associated with $\rho(\mathcal{A})$.

Let $\mathbb{R}_{+}^{n}$ denote the set of $n$-dimensional nonnegative real vectors. A tensor is called symmetric if its entries are invariant under any permutation of their indices.
\begin{lemma} \label{lem Qi} \textup{\cite{Qi2013}}
Suppose that $\mathcal{A}=(a_{i_1\cdots i_m})$ is an order $m$ dimension $n$ symmetric nonnegative tensor, with $m\geq2$. Then the spectral radius of $\mathcal{A}$ is equal to
\begin{eqnarray*}
\max\left\{\mathcal{A}x^m: \sum_{i=1}^{n}x_{i}^{m}=1, x=(x_{1},\ldots,x_{n})^\top\in\mathbb{R}_{+}^{n}\right\},
\end{eqnarray*}
where $\mathcal{A}x^m=x^\top(\mathcal{A}x^{m-1})=\sum_{i_1,\ldots,i_m=1}^{n}a_{i_1\cdots i_m}x_{i_1}\cdots x_{i_m}$.
\end{lemma}

\begin{lemma}\label{lem2.2}
\textup{\cite[Theorem 5.3]{YY}} Let $\mathcal{A}$ be an $m$-order $n$-dimension nonnegative tensor. Then
\begin{eqnarray*}
\rho(\mathcal{A})=\max_{0\neq x\in\mathbb{R}_+^n}\min_{x_i>0}\frac{(\mathcal{A}x^{m-1})_i}{x_i^{m-1}}.
\end{eqnarray*}

\end{lemma}
For an $n$-vertex $r$-uniform hypergraph $H$, by Lemma \ref{lem Qi}, we know that
\begin{equation}
\rho(H)=\max\left\{\mathcal{A}_Hx^r=r\sum_{i_1\cdots i_r\in E(H)}x_{i_1}\cdots x_{i_1}: \sum_{i\in V(H)}x_i^r=1, x\in\mathbb{R}_+^n\right\}.\tag{2.1}
\end{equation}
For a vertex $u$ of $H$, the degree of $u$ is the number of edges containing $u$ in $H$. Let $\Delta(H)$ denote the maximum degree of $H$.
\begin{lemma}\label{lem2.3}
\textup{\cite[Theorem 3.8]{Cooper}} Let $H$ be an $n$-vertex $r$-uniform hypergraph. Then
\begin{eqnarray*}
\frac{r|E(H)|}{n}\leq\rho(H)\leq\Delta(H).
\end{eqnarray*}
\end{lemma}
Let $K_{r+1}^+$ denote the $r$-uniform hypergraph obtained from $K_{r+1}$ by enlarging each edge of $K_{r+1}$ with $r-2$ new vertices such that distinct edges of $K_{r+1}$ are enlarged by distinct vertices. An $r$-uniform hypergraph is called \textit{linear} if every two edges have at most one common vertex.
\begin{lemma}\label{lem2.4}
\textup{\cite[Theorem 1.1]{GaoChang}} Let $H$ be a $K_{r+1}^+$-free $r$-uniform linear hypergraph with $n$ vertices. Then $\rho(H)\leq\frac{n}{r}$ when $n$ is sufficiently large.
\end{lemma}

\begin{lemma}\label{lem2.5}
\textup{\cite{Liu}} For a graph $G$ on $n$ vertices, we have
\begin{eqnarray*}
|C_r(G)|\leq\frac{n}{r}\rho_r(G).
\end{eqnarray*}
\end{lemma}
The \textit{clique number} of a graph $G$, denoted by $\omega(G)$, is the size of the maximum clique in $G$.
\begin{lemma}\label{lem2.6}
\textup{\cite{LiuBu}} Let $G$ be an $n$-vertex graph. For any positive integer $t\leq\omega(G)$, we have
\begin{equation*}
\max\left\{\sum_{\{i_1,\ldots,i_t\}\in C_t(G)}x_{i_1}\cdots x_{i_t}:\sum_{i=1}^nx_i=1,x\in\mathbb{R}^n_+\right\}={\omega(G)\choose t}\omega(G)^{-t}.
\end{equation*}
\end{lemma}

An \textit{orthonormal representation} of an $n$-vertex graph $G$ is a set $\{u_1,\ldots,u_n\}$ of unit real vectors such that $u_i^\top u_j=0$ if $i$ and $j$ are two nonadjacent vertices in $G$. For an orthonormal representation $f=\{u_1,\ldots,u_n\}$ of a graph $G$, its Gram matrix $M_f$ is the positive semidefinite matrix such that $(M_f)_{ij}=u_i^\top u_j$. The orthonormal representation $f$ is called nonnegative if $u_i^\top u_j\geq0$ for any $i,j\in V(G)$.

The following spectral characterization of the Schrijver theta function $\vartheta'(G)$ follows from the equation (A.6) in \cite{Acin}.
\begin{lemma}\label{lem2.7}
\textup{\cite{Acin}} For any graph $G$, $\vartheta'(G)$ is the maximum, over all nonnegative orthonormal representations $f$ of the complement graph $\overline{G}$, of the largest eigenvalue of the Gram matrix $M_f$.
\end{lemma}

Given a graph $G$ and a real number $k\geq2$,  a vector $k$-coloring of $G$ is an assignment of unit vectors $u_i$ to each vertex $i\in V(G)$ such that $u_i^\top u_j\leq-\frac{1}{k-1}$ for all edge $ij\in E(G)$. The \textit{vector chromatic number} \cite{Karger} $\chi_v(G)$ is the smallest real number $k$ for which a vector $k$-coloring exists.
\begin{lemma}\label{lem2.8}
\textup{\cite{Balla}} For any graph $G$, we have $\chi_v(G)=\vartheta'(\overline{G})$.
\end{lemma}

\section{Local structure and spectral radius of hypergraphs}
For a vertex $u$ of an $r$-uniform hypergraph $H$, let $N(u)=\{v:vu\in E(H^{(2)})\}$ denote the open neighborhood of $u$ in the $2$-section graph $H^{(2)}$. Let $H(u)$ denote the $(r-1)$-uniform hypergraph with vertex set $N(u)$ and edge set $\{e\setminus\{u\}:e\in E_u(H)\}$, where $E_u(H)$ denotes the set of edges containing $u$ in $H$.

The spectral radius ratio $\beta(H)=\frac{\rho(H)}{\rho(H^{(2)})}$ and local hypergraph $H(u)$ have the following relationship.
\begin{theorem}\label{thm3.1}
Let $H$ be an $r$-uniform hypergraph ($r\geq3$) such that $|E(H)|\neq\emptyset$. Then there exists a vertex $u$ such that
\begin{eqnarray*}
\rho(H(u))\geq(r-1)\beta(H)
\end{eqnarray*}
and $H$ has $p=\lceil(r-1)\alpha(H)\rceil\geq\left\lceil\frac{r(r-1)|E(H)|}{|V(H)|\rho(H^{(2)})}\right\rceil$ distinct edges $e_1,\ldots,e_p$ such that $u\in e_1\cap\cdots\cap e_p$ and
\begin{eqnarray*}
|e_1\cap\cdots\cap e_p|\geq2.
\end{eqnarray*}
\end{theorem}
\begin{proof}
Let $x$ be a nonnegative eigenvector associated with $\rho(H)$, then $\mathcal{A}_Hx^{r-1}=\rho(H)x^{[r-1]}$. So for any vertex $i$ of $H$, we have
\begin{eqnarray*}
\rho(H)x_i^{r-1}=(\mathcal{A}_Hx^{r-1})_i=\sum_{j_1\cdots j_{r-1}i\in E(H)}x_{j_1}\cdots x_{j_{r-1}}.
\end{eqnarray*}
Let $y$ be the $|N(i)|$-dimension vector such that $y_j=x_j$ for all $j\in N(i)$. Then by the equation (2.1), we obtain
\begin{eqnarray*}
\rho(H)x_i^{r-1}&=&\sum_{j_1\cdots j_{r-1}i\in E(H)}x_{j_1}\cdots x_{j_{r-1}}=(r-1)^{-1}\mathcal{A}_{H(i)}y^{r-1}\\
&\leq&(r-1)^{-1}\rho(H(i))\sum_{j\in N(i)}x_j^{r-1}.
\end{eqnarray*}
Let $z$ be the $|V(G)|$-dimension nonnegative vector such that $z_j=x_j^{r-1}$ for all $j\in V(G)$. Then $\sum_{j\in N(i)}z_j=(\mathcal{A}_{H^{(2)}}z)_i$, where $\mathcal{A}_{H^{(2)}}$ is the adjacency matrix of $H^{(2)}$. By Lemma \ref{lem2.2}, we know that there exists a vertex $u$ such that
\begin{eqnarray*}
x_u^{-(r-1)}\sum_{j\in N(u)}x_j^{r-1}=z_u^{-1}\sum_{j\in N(u)}z_j\leq\rho(H^{(2)}).
\end{eqnarray*}
Hence
\begin{eqnarray*}
\rho(H)&\leq&(r-1)^{-1}\rho(H(u))\rho(H^{(2)}),\\
\rho(H(u))&\geq&\frac{(r-1)\rho(H)}{\rho(H^{(2)})}=(r-1)\beta(H).
\end{eqnarray*}
By Lemma \ref{lem2.3}, we know that the maximum degree $\Delta$ of $H(u)$ satisfies
\begin{eqnarray*}
\Delta\geq\rho(H(u))\geq(r-1)\beta(H).
\end{eqnarray*}
So we get
\begin{eqnarray*}
\Delta\geq p=\lceil(r-1)\beta(H)\rceil.
\end{eqnarray*}
Let $v$ a vertex with maximum degree in $H(u)$, then $H$ has $p$ distinct edges $e_1,\ldots,e_p$ such that
\begin{eqnarray*}
\{u,v\}\subseteq e_1\cap\cdots\cap e_p.
\end{eqnarray*}
By Lemma \ref{lem2.3}, we have
\begin{eqnarray*}
p=\left\lceil\frac{(r-1)\rho(H)}{\rho(H^{(2)})}\right\rceil\geq\left\lceil\frac{r(r-1)|E(H)|}{|V(H)|\rho(H^{(2)})}\right\rceil.
\end{eqnarray*}
\end{proof}
For an $r$-uniform hypergraph $F$, let $K_1\ast F$ denote the $(r+1)$-uniform hypergraph with vertex set $V(F)\cup\{x\}$ and edge set $\{xu_1\cdots u_r:u_1\cdots u_r\in E(F)\}$. For a set $\mathcal{H}$ of $r$-uniform hypergraphs, let $spex_\mathcal{H}(n,F)$ denote the maximum spectral radius of an $n$-vertex $F$-free hypergraph in $\mathcal{H}$.
\begin{theorem}\label{thm3.2}
Let $H$ be an $(r+1)$-uniform hypergraph ($r\geq2$) with at least one edge, and let $\mathcal{H}$ be a set of $r$-uniform hypergraphs such that $H(i)\in\mathcal{H}$ for each vertex $i\in V(H)$. Suppose that $F$ is an $r$-uniform hypergraph without isolated vertices. Then $K_1\ast F$ is a subhypergraph of $H$ if one the following holds:\\
(1) $\beta(H)>r^{-1}spex_\mathcal{H}\left(\max_{i\in V(H)}|N(i)|,F\right)$.\\
(2) $(r+1)|E(H)|>r^{-1}|V(H)|\rho(H^{(2)})spex_\mathcal{H}\left(\max_{i\in V(H)}|N(i)|,F\right)$.\\
(3) $\rho(H)>r^{-1}\Delta(H^{(2)})spex_\mathcal{H}\left(\max_{i\in V(H)}|N(i)|,F\right)$.
\end{theorem}
\begin{proof}
If $\beta(H)>r^{-1}spex_\mathcal{H}\left(\max_{i\in V(H)}|N(i)|,F\right)$, then by Theorem \ref{thm3.1}, there exists a vertex $u$ such that
\begin{eqnarray*}
\rho(H(u))\geq r\beta(H)>spex_\mathcal{H}\left(\max_{i\in V(H)}|N(i)|,F\right)\geq spex_\mathcal{H}(|N(u)|,F).
\end{eqnarray*}
Since $H(u)\in\mathcal{H}$ and $H(u)$ has $|N(u)|$ vertices, $F$ is a subhypergraph of $H(u)$, that is, $K_1\ast F$ is a subhypergraph of $H$. Hence part (1) holds.

By Lemma \ref{lem2.3}, we know that $\beta(H)=\frac{\rho(H)}{\rho(H^{(2)})}\geq\frac{(r+1)|E(H)|}{|V(H)|\rho(H^{(2)})}$ and $\beta(H)\geq\frac{\rho(H)}{\Delta(H^{(2)})}$. So parts (2) and (3) hold.
\end{proof}
We can derive the following result from Theorem \ref{thm3.2}.
\begin{corollary}
Let $H$ be an $(r+1)$-uniform hypergraph ($r\geq2$) such that $|e\cap f|\leq2$ for any two distinct edges $e,f\in E(H)$. If $\max_{u\in V(H)}|N(u)|$ is sufficiently large with respect to $r$ and $\beta(H)>r^{-2}\max_{u\in V(H)}|N(u)|$, then $K_1\ast K_{r+1}^+$ is a subhypergraph of $H$.
\end{corollary}
\begin{proof}
Let $\mathcal{H}$ be the set of $r$-uniform linear hypergraphs. Since $|e\cap f|\leq2$ for any two distinct edges $e,f\in E(H)$, we know that $H(i)\in\mathcal{H}$ for each vertex $i\in V(H)$. By Lemma \ref{lem2.4}, we have
\begin{eqnarray*}
spex_\mathcal{H}\left(\max_{i\in V(H)}|N(i)|,K_{r+1}^+\right)\leq r^{-1}\max_{u\in V(H)}|N(u)|
\end{eqnarray*}
when $\max_{u\in V(H)}|N(u)|$ is sufficiently large. If $\beta(H)>r^{-2}\max_{u\in V(H)}|N(u)|$, then
\begin{eqnarray*}
\beta(H)>r^{-1}spex_\mathcal{H}\left(\max_{i\in V(H)}|N(i)|,K_{r+1}^+\right).
\end{eqnarray*}
By Theorem \ref{thm3.2}, we know that $K_1\ast K_{r+1}^+$ is a subhypergraph of $H$.
\end{proof}
The \textit{Lagrangian} \cite{Hefetz} of an $r$-uniform hypergraph $H$ is defined as
\begin{eqnarray*}
\lambda(H)=\max\left\{\sum_{j_1\cdots j_r\in E(H)}x_{j_1}\cdots x_{j_r}:\sum_{i\in V(H)}x_i=1, x\in\mathbb{R}_{+}^{|V(H)|}\right\}.
\end{eqnarray*}
Given an $r$-uniform hypergraph $F$, the \textit{Lagrangian density} \cite{JiangPeng} $\pi_\lambda(F)$ of $F$ is
\begin{eqnarray*}
\pi_\lambda(F)=\sup\{r!\lambda(H): H~\textup{is $r$-uniform and \textit{F}-free}\}.
\end{eqnarray*}
The spectral radius $\rho(H)$ and Lagrangian of local hypergraph $H(u)$ have the following relationship.
\begin{theorem}\label{thm3.4}
Let $H$ be an $r$-uniform hypergraph ($r\geq3$) such that $|E(H)|\neq\emptyset$. Then there exists a vertex $u$ in $H$ such that
\begin{eqnarray*}
\rho(H)\leq\rho(H^{(2)})^{r-1}\lambda(H(u)).
\end{eqnarray*}
Moreover, for an $(r-1)$-uniform hypergraph $F$, $K_1\ast F$ is a subhypergraph of $H$ if one of the following holds:\\
(1) $(r-1)!\rho(H)>\rho(H^{(2)})^{r-1}\pi_\lambda(F)$.\\
(2) $r!|E(H)|>|V(H)|\rho(H^{(2)})^{r-1}\pi_\lambda(F)$.\\
(3) $(r-1)!\rho(H)>\Delta(H^{(2)})^{r-1}\pi_\lambda(F)$.
\end{theorem}
\begin{proof}
Since $|E(H)|\neq\emptyset$, we have $\rho(H)>0$. Let $x$ be a nonnegative eigenvector associated with $\rho(H)$, then $\mathcal{A}_Hx^{r-1}=\rho(H)x^{[r-1]}$. For any vertex $i$ satisfying $x_i>0$, we have
\begin{eqnarray*}
\rho(H)x_i^{r-1}=(\mathcal{A}_Hx^{r-1})_i=\sum_{j_1\cdots j_{r-1}i\in E(H)}x_{j_1}\cdots x_{j_{r-1}}
\end{eqnarray*}
and $\sum_{u\in N(i)}x_u>0$. Let $y$ be the $|N(i)|$-dimension nonnegative vector such that $y_j=x_j(\sum_{u\in N(i)}x_u)^{-1}$ for all $j\in N(i)$. Then $\sum_{j\in N(i)}y_j=1$ and
\begin{eqnarray*}
\rho(H)x_i^{r-1}=\left(\sum_{j\in N(i)}x_j\right)^{r-1}\sum_{j_1\cdots j_{r-1}i\in E(H)}y_{j_1}\cdots y_{j_{r-1}}\leq\left(\sum_{j\in N(i)}x_j\right)^{r-1}\lambda(H(i)).
\end{eqnarray*}
Notice that $\sum_{j\in N(i)}x_j=(\mathcal{A}_{H^{(2)}}x)_i$. By Lemma \ref{lem2.2}, we know that there exists a vertex $u$ such that
\begin{eqnarray*}
x_u^{-1}\sum_{j\in N(u)}x_j\leq\rho(H^{(2)}).
\end{eqnarray*}
Hence
\begin{eqnarray*}
\rho(H)\leq\rho(H^{(2)})^{r-1}\lambda(H(u)).
\end{eqnarray*}

Moreover, if $F$ is an $(r-1)$-uniform hypergraph and $(r-1)!\rho(H)>\rho(H^{(2)})^{r-1}\pi_\lambda(F)$, then
\begin{eqnarray*}
(r-1)!\lambda(H(u))\geq\frac{(r-1)!\rho(H)}{\rho(H^{(2)})^{r-1}}>\pi_\lambda(F).
\end{eqnarray*}
Since $\pi_\lambda(F)=\sup\{(r-1)!\lambda(G): G~\textup{is $(r-1)$-uniform and \textit{F}-free}\}$, we know that $F$ is a subhypergraph of $H(u)$. Hence $K_1\ast F$ is a subhypergraph of $H$. So part (1) holds.

By Lemma \ref{lem2.3}, we know that $\rho(H)\geq\frac{r|E(H)|}{|V(H)|}$ and $\rho(H^{(2)})\leq\Delta(H^{(2)})$. So parts (2) and (3) hold.
\end{proof}
Theorem \ref{thm3.4} implies that certain local subhypergraph will occur when $\frac{\rho(H)}{\rho(H^{(2)})^{r-1}}$ is large enough. The following is an example.
\begin{example}
The $r$-uniform $t$-matching, denoted by $M_t^r$, is the $r$-uniform hypergraph with $t$ pairwise disjoint edges. It is known that $\pi_\lambda(M_t^3)=3!{3t-1\choose3}(3t-1)^{-3}$ for $t\geq2$ (see \cite[Corollary 3.4]{JiangPeng}). For a $4$-uniform hypergraph $H$, if $\frac{\rho(H)}{\rho(H^{(2)})^3}>{3t-1\choose3}(3t-1)^{-3}$, then by Theorem \ref{thm3.4}, we know that $K_1\ast M_t^3$ is a subhypergraph of $H$.
\end{example}

\section{Local structure and $r$-clique spectral radius of graphs}

Recall that $\gamma_r(G)=\frac{\rho_r(G)}{\rho_2(H)}$, where $H$ is the subgraph of a graph $G$ obtained by deleting all edges that are not containing in an $r$-clique. Clearly, we have $\rho_r(G)=\rho(H_r)$ and $\gamma_r(G)=\beta(H_r)$, where $H_r$ is the $r$-uniform hypergraph with vertex set $V(H_r)=V(G)$ and edge set $E(H)=\{i_1\cdots i_r:\{i_1,\ldots,i_r\}\in C_r(G)\}$.

For a vertex $u$ of $G$, let $N_G(u)=\{v:vu\in E(G)\}$ denote the open neighborhood of $u$ in $G$, and let $G_u$ denote the subgraph induced by $N_G(u)$. The following result follows from Theorem \ref{thm3.1}.
\begin{theorem}\label{thm4.1}
Let $G$ be a graph such that $3\leq r\leq\omega(G)$ for an integer $r$. Then there exists a vertex $u$ in $G$ such that
\begin{eqnarray*}
\rho_{r-1}(G_u)\geq(r-1)\gamma_r(G)
\end{eqnarray*}
and $G$ has $p=\left\lceil(r-1)\beta_r(G)\right\rceil$ distinct $r$-cliques $C_1,\ldots,C_p$ such that $u\in C_1\cap\cdots\cap C_p$ and
\begin{eqnarray*}
|C_1\cap\cdots\cap C_p|\geq2.
\end{eqnarray*}
\end{theorem}

The book graph $B_p$ is the graph that consists of $p$ triangles sharing a common edge. Take $r=3$ in Theorem \ref{thm4.1}, we obtain the following result, which shows that a graph $G$ has a large book subgraph when $\gamma_3(G)$ is large.
\begin{corollary}
Let $G$ a graph with at least one triangle. Then $B_p$ is a subgraph of $G$, where
\begin{eqnarray*}
p=\left\lceil2\gamma_3(G)\right\rceil\geq\left\lceil\frac{2\rho_3(G)}{\rho_2(G)}\right\rceil.
\end{eqnarray*}
\end{corollary}
For a graph $F$, let $K_1\vee F$ denote the graph with vertex set $V(F)\cup\{x\}$ and edge set $E(F)\cup\{xu:u\in V(F)\}$, and let $spex(n,F)$ denote the maximum spectral radius of an $n$-vertex $F$-free graph. Let $\Delta(H)$ denote the maximum degree of a graph $H$.
\begin{theorem}
Let $G$ a graph with at least one triangle, and let $H$ be the subgraph of $G$ obtained by deleting all edges that are not containing in a triangle. Suppose that $F$ is a graph without isolated vertices. Then $K_1\vee F$ is a subgraph of $G$ if one of the following holds:\\
(1) $\gamma_3(G)>\frac{1}{2}spex\left(\Delta(H),F\right)$.\\
(2) $6|C_3(G)|>|V(G)|\rho_2(H)spex\left(\Delta(H),F\right)$.\\
(3) $\rho_3(G)>\frac{1}{2}\Delta(H)spex\left(\Delta(H),F\right)$.
\end{theorem}
\begin{proof}
If $\gamma_3(G)>\frac{1}{2}spex\left(\Delta(H),F\right)$, then by Theorem \ref{thm4.1}, there exists a vertex $u$ such that
\begin{eqnarray*}
\rho_2(H_u)=\rho_2(G_u)\geq2\gamma_3(G)>spex\left(\Delta(H),F\right)\geq spex_\mathcal{H}(|V(H_u)|,F).
\end{eqnarray*}
Hence $F$ is a subgraph of $H_u$, that is, $K_1\vee F$ is a subgraph of $G$. So part (1) holds.

By Lemma \ref{lem2.5}, we have $\gamma_3(G)=\frac{\rho_3(G)}{\rho_2(H)}\geq\frac{3c_3(G)}{|V(G)|\rho_2(H)}$, so part (2) holds. Since $\gamma_3(G)\geq\frac{\rho_3(G)}{\Delta(H)}$, part (3) holds.
\end{proof}
From Theorem \ref{thm3.4}, we can derive the following relation between the spectral radius and clique number of a graph.
\begin{theorem}
For any graph $G$ and any integer $r$ satisfying $3\leq r\leq\omega(G)$, we have
\begin{eqnarray*}
\rho_r(G)\leq\rho(\widetilde{G})^{r-1}{\omega(G)-1\choose r-1}(\omega(G)-1)^{-(r-1)}.
\end{eqnarray*}
where $\widetilde{G}$ is the subgraph of $G$ obtained by deleting all edges that are not containing in an $r$-clique.
\end{theorem}
\begin{proof}
Let $H$ be the $r$-uniform hypergraph with vertex set $V(H)=V(G)$ and edge set $E(H)=\{i_1\cdots i_r:\{i_1,\ldots,i_r\}\in C_r(G)\}$. Then $\rho_r(G)=\rho(H)$ and $\widetilde{G}=H^{(2)}$ is the $2$-section graph of $H$. By Theorem \ref{thm3.4}, there exists a vertex $u$ in $H$ such that
\begin{eqnarray*}
\rho_r(G)\leq\rho(\widetilde{G})^{r-1}\lambda(H(u)).
\end{eqnarray*}
Note that $\lambda(H(u))=\max\left\{\sum_{\{j_1,\ldots,j_{r-1}\}\in C_{r-1}(G_u)}x_{j_1}\cdots x_{j_{r-1}}:\sum_{i\in N(u)}x_i=1, x\in\mathbb{R}_{+}^{|N(u)|}\right\}$. Since $\omega(G_u)\leq\omega(G)-1$, by Lemma \ref{lem2.6}, we have
\begin{eqnarray*}
\lambda(H(u))={\omega(G_u)\choose r-1}\omega(G_u)^{-(r-1)}\leq{\omega(G)-1\choose r-1}(\omega(G)-1)^{-(r-1)}.
\end{eqnarray*}
\end{proof}

\section{Local chromatic number and graph eigenvalues}

In \cite{ErdosFuredi}, Erd\H{o}s et al. defined the \textit{local chromatic number} $\psi(G)$ of a graph $G$ as
\begin{eqnarray*}
\psi(G)=1+\min_c\max_{v\in V(G)}|\{c(u):u\in N_G(v)\}|,
\end{eqnarray*}
where the minimum is taken over all proper colorings $c$ of $G$. Then $\psi(G)$ is a local graph parameter based on graph colorings, and $\psi(G)$ is bounded from below by the fractional chromatic number of $G$ (see \cite{Korner}).

Considering the localized problem of vector chromatic number, it is natural to define the \textit{local vector chromatic number} of $G$ as
\begin{eqnarray*}
\phi(G)=1+\max_{u\in V(G)}\chi_v(G_u).
\end{eqnarray*}
Then it is easy to obtain the following relations
\begin{eqnarray*}
\omega(G)\leq\phi(G)\leq1+\max_{u\in V(G)}\chi(G_u)\leq\psi(G)\leq\chi(G).
\end{eqnarray*}

For a graph $G$ with vertex set $\{1,\ldots,n\}$, we say that a set $f=\{u_1,\ldots,u_n\}$ of unit real vectors is a \textit{$2$-distance representation} of $G$ if the following hold:\\
(1) $u_i^\top u_j=0$ whenever the distance between vertices $i$ and $j$ is $2$.\\
(2) $u_i^\top u_j\geq0$ for any $1\leq i\leq j\leq n$.

For a $2$-distance representation $f=\{u_1,\ldots,u_n\}$ of $G$, let $\mathcal{A}_f(G)=(a_{ijk})$ denote the $3$-order $n$-dimension \textit{weighted $3$-clique tensor} of $G$ with entries
\begin{align*}
a_{ijk}=
\begin{cases}
   u_j^\top u_k, &\{i,j,k\}\in C_3(G).\\
   0, &\textup{otherwise}.
\end{cases}
\end{align*}
Let $\mu_f(G)$ denote the spectral radius of $\mathcal{A}_f(G)$.
\begin{theorem}
Let $G$ be a graph with at least one triangle, and let $H$ be the subgraph of $G$ obtained by deleting all edges that are not containing in a triangle. For any $2$-distance representation $f$ of $G$, we have
\begin{eqnarray*}
\phi(G)\geq2+\left\lceil\frac{\mu_f(G)}{\rho_2(H)}\right\rceil.
\end{eqnarray*}
\end{theorem}
\begin{proof}
Suppose that $f=\{u_1,\ldots,u_n\}$ is a $2$-distance representation of $G$. Let $x$ be a nonnegative eigenvector associated with $\mu_f(G)$, then $\mathcal{A}_f(G)x^2=\mu_f(G)x^{[2]}$. For any vertex $i$ of $G$, we have
\begin{eqnarray*}
\mu_f(G)x_i^2=(\mathcal{A}_f(G)x^2)_i=2\sum_{jk\in E(H_i)}u_j^\top u_kx_jx_k.
\end{eqnarray*}
Note that $u_j^\top u_k=0$ whenever $i,k$ are two distinct non-adjacent vertices in $H_i$. So $\{u_j:j\in N_H(i)\}$ is a nonnegative orthonormal representation of $H_i$. Let $y$ be the $|N_H(i)|$-dimension vector such that $y_j=x_j$ for all $j\in N_H(i)$. Then
\begin{eqnarray*}
\mu_f(G)x_i^2&=&2\sum_{jk\in E(H_i)}u_j^\top u_kx_jx_k=y^\top(M-I)y\\
&\leq&(\lambda_1(M)-1)\sum_{j\in N_H(i)}x_j^2,
\end{eqnarray*}
where $M$ is the Gram matrix of $\{u_j:j\in N_H(i)\}$, $\lambda_1(M)$ is the maximum eigenvalue of $M$. By Lemmas \ref{lem2.7} and \ref{lem2.8}, we have
\begin{eqnarray*}
\mu_f(G)x_i^2&\leq&(\lambda_1(M)-1)\sum_{j\in N_H(i)}x_j^2\leq(\chi_v(H_i)-1)\sum_{j\in N_H(i)}x_j^2,\\
&\leq&(\chi_v(G_i)-1)\sum_{j\in N_H(i)}x_j^2.
\end{eqnarray*}
Let $z$ be the $|V(G)|$-dimension vector such that $z_j=x_j^2$ for all $j\in V(G)$. Then $\sum_{j\in N_H(i)}x_j^2=(A_Hz)_i$. By Lemma \ref{lem2.2}, we know that there exists a vertex $u$ such that
\begin{eqnarray*}
x_u^{-2}\sum_{j\in N_H(u)}x_j^2=z_u^{-1}\sum_{j\in N_H(u)}z_j\leq\rho_2(H).
\end{eqnarray*}
Since $\phi(G)=1+\max_{i\in V(G)}\chi_v(G_i)$, we obtain
\begin{eqnarray*}
\mu_f(G)&\leq&(\chi_v(G_u)-1)\rho_2(H)\leq(\phi(G)-2)\rho_2(H),\\
\phi(G)&\geq&2+\left\lceil\frac{\mu_f(G)}{\rho_2(H)}\right\rceil.
\end{eqnarray*}
\end{proof}





\end{CJK*}
\end{spacing}
\end{document}